\documentclass[12pt,a4paper]{amsart}
\usepackage{amsmath,amsfonts,amsthm,amsopn,color,amssymb,enumitem}
\usepackage[a4paper]{geometry}
\usepackage{palatino}
\usepackage{graphicx}
\usepackage[colorlinks=true]{hyperref}
\hypersetup{urlcolor=blue, citecolor=red, linkcolor=blue}
\usepackage{amsmath,tikz}
\usetikzlibrary{calc}

{\left\lbrace\begin{array}{@{}l@{}}}%
{\end{array}\right.}

\usepackage{cite}
\usepackage{relsize}
\usepackage{esint}
\usepackage{verbatim}
\usepackage{mathrsfs}
\usepackage{xcolor}
\usepackage{tikz}
\usepackage{pgfplots}

\newcommand{\e}{\varepsilon}

\newcommand{\N}{\mathbb{N}}

\renewcommand{\d }{\delta }

\newcommand{\beq}{\begin{equation}}
\newcommand{\eeq}{\end{equation}}

\newcommand{\ee}{{\mathtt e}}

\newtheorem{theorem}{Theorem}[section]
\newtheorem*{theorem*}{Theorem}
\newtheorem{lemma}[theorem]{Lemma}

\newtheorem{definition}[theorem]{Definition}
\newtheorem{proposition}[theorem]{Proposition}

\theoremstyle{definition}
\newtheorem{remark}[theorem]{Remark}

\author[G. Mancini]{Gabriele Mancini}
\address{\noindent  Dipartimento di Matematica, Universit\'a degli Studi di Bari Aldo Moro,Italy }
\email{gabriele.mancini@uniba.it}

\author[G. M. Rago]{Giuseppe Mario Rago}
\address{\noindent  Dipartimento di Matematica, Universit\'a degli Studi di Bari Aldo Moro,Italy }
\email{g.rago6@phd.uniba.it}

\author[G. Vaira]{Giusi Vaira}
\address{\noindent  Dipartimento di Matematica, Universit\'a degli Studi di Bari Aldo Moro,Italy }
\email{giusi.vaira@uniba.it}
\thanks{Work partially supported by 
the MUR-PRIN-P2022YFAJH ``Linear and Nonlinear PDE’s: New directions and Applications" and by the INdAM-GNAMPA projects 2025 ``Fenomeni non lineari: problemi locali e non locali e loro applicazioni" and ``{Critical and limiting phenomena in nonlinear elliptic systems}" CUP E5324001950001. The authors thank PNRR MUR project CN00000013 HUB - National Centre for HPC, Big Data and Quantum Computing (CUP H93C22000450007).}

\subjclass{35B44, 35B33, 35J25, 35J08}
\keywords{Critical problems; Brezis-Nirenberg; blowing-up solutions; annulus; Green function; Gegenbauer Polynomials.}

\begin{document}
\title[Almost critical problems on an annulus]{Existence of positive solutions for a class of almost critical problems on an annulus}
\maketitle 
\begin{abstract}
In this paper we will consider multi-peaks positive solutions for a class of slightly subcritical or slightly supercritical elliptic problems on an annulus with Dirichlet boundary conditions. By using the explicit form of the Green function and of the Robin function on the annulus, we prove that the annulus becomes thinner and thinner when the number of bumps increases for  the slightly subcritical case, while the hole of the annulus is very small for the slightly supercritical case.
\end{abstract}

\section{Introduction}
In the last decades, there has been considerable interest in the study of positive solutions to the following critical problem 
\beq\label{p1}
\left\{\begin{aligned} &-\Delta u =u^{2^*-1} \quad &\mbox{in}\,\, \Omega\,\\
& u>0\quad &\mbox{in}\,\, \Omega\,\\
&u=0\quad &\mbox{on}\,\, \partial\Omega,\end{aligned}\right.\eeq
 where $\Omega\subset \mathbb R^N$ with $N\geq 3$  is a smooth and bounded domain and $2^*:=\frac{2N}{N-2}$ is the critical exponent for the well known Sobolev embedding.\\
 The interest in this equation relies on many physical and geometrical motivations. For example, \eqref{p1} is a special case of the classical Yamabe problem in differential geometry and it is related to some mathematical models in astrophysics (see \cite{Horedt} for physical motivations).\\ From the mathematical point of view the interest comes from the fact that, despite the simple form of problem \eqref{p1}, the structure of the solution set is surprisingly rich and very complex. This is why a vast literature on \eqref{p1} is available regarding existence, multiplicity and qualitative behavior of solutions. Even if many contributions are given in the last decades, some basic issues are still open.\\
The criticality of the problem \eqref{p1} affects the solvability question: for instance, Pohozaev \cite{P} proved that \eqref{p1} has no solutions under the assumption that $\Omega$ is star-shaped.\\
 On the other hand, Kazdan and Warner showed in \cite{KW} that if $\Omega$ is an annulus then \eqref{p1} has a (unique) positive solution in the class of functions with radial symmetry. In \cite{BP}, the authors study also the asymptotic
behavior of this solution as the inner radius of the annulus tends to zero.\\ From the paper of Kazdan and Warner comes the idea of exploiting the topology of $\Omega$ in order to get solutions for \eqref{p1}. Indeed,  in the non-symmetric case,
Coron \cite{Coron} found via variational methods that \eqref{p1} is solvable and that it admits a positive solution under the assumption that $\Omega$ is a domain with a small hole. \\ After that, a substantial improvement of this result was obtained by Bahri and Coron \cite{BC}, showing that if some homology group of $\Omega$ with coefficients in $\mathbf Z_2$ is not trivial, then \eqref{p1} has at least one positive solution. Nevertheless, Ding \cite{Ding} (see also Dancer \cite{Dancer}) gave the example of contractible domains on which a solution still exists, showing that both the topology and the geometry of the domain play a crucial role. \\
An interesting question is the study of the asymptotic behavior of Coron’s solution as the size of the hole tends to zero. Under the assumption that the hole is symmetric (a ball of radius $\rho$), then the solution concentrates around
the hole and it converges in the sense of measure, as $\rho\to 0$, to a Dirac delta centered at the center of the hole. In the
literature this is what is known as a (simple) bubbling solution. We refer the reader to \cite{Lewandowski, LYY, Rey1} where the study of
existence of positive solutions to \eqref{p1} in domains with several small symmetric holes and their asymptotic behavior
as the size of the holes goes to zero is carried out.\\
Also the transition between slightly subcritical power and slightly supercritical power implies a great change in the structure of the solution set of problem \eqref{p1}.\\ Specifically, we let the problems

\[
(\mathcal P_\pm^\e)\quad \begin{cases}\begin{aligned} &-\Delta u =u^{2^*-1\pm\varepsilon} \quad &\mbox{in}\,\, \Omega\,\\
& u>0\quad &\mbox{in}\,\, \Omega\,\\
&u=0\quad &\mbox{on}\,\, \partial\Omega,\end{aligned}
\end{cases}
\]
where $\varepsilon>0$ is small.\\
Solvability in the two regimes is very different. For instance, for problem $(\mathcal P_+^\e)$ the variational approaches are not available and the Pohozaev's result can be applied again giving no existence at least in a star-shaped domain. \\ Also the Kazdan and Warner's result \cite{KW} holds and hence in a symmetric annulus a radial solution can be again obtained by a minimization argument.\\ 
Hence, also in this case, the topology and the geometry of the domain are crucial conditions for the existence result, even though finding general conditions is a well-known open problem. Indeed, a question raised by Brèzis in \cite{Brezis} is whether the presence of a nontrivial topology in the domain suffices for solvability in the supercritical case (with a power $p>2^*-1$). The answer to this question is negative since Passaseo in \cite{Passaseo1, Passaseo2} constructed domains in dimension $N\geq 4$ for which no solutions exist for powers $p$ near the second critical exponent $\frac{N+1}{N-3}$. 
The nontrivial topology is, however, a necessary condition as shown by Molle and Passaseo \cite{MP}. 
Instead, for problem $(\mathcal P_-^\e)$ the least energy solution (that can be found by using variational arguments) exists and, as $\e\to 0^+$, blows-up at a  critical point of the Robin function of the domain $\Omega$. Moreover, under suitable conditions on the shape of the domain, several papers deal with multiple blow-up solutions (see \cite{BLR, BrP, Han, FW, Rey2, Rey3}). However, in a convex domain, multiple blow-up phenomenon cannot occur. Grossi and Takahashi \cite{GT}
 proved the nonexistence of positive solutions for the problem $(\mathcal P_-^\e)$ blowing up at more than one point. In contrast, the single bubble phenomenon does not appear  for problem $(\mathcal P_+^\e)$ as $\e\to 0^+$ as it is shown by Ben Ayed, El Mehdi, Grossi and Rey in \cite{BEGR}. However, multi-bubble solutions have been found in a domain with small holes both in the symmetric case and in the non symmetric case (see \cite{Daprile, DFM1, DFM2, PR}).
\\ In particular, in \cite{Daprile}, by using a finite dimensional variational reduction approach, the author constructs a family of blowing up positive solutions for $(\mathcal P_+^\e)$ when 
the domain $\Omega\subset \mathbb R^N$ has a non symmetric small hole with $N\geq 3$. Indeed, the solution looks like \beq\label{approx}u_\e(x) \sim \sum_{i=1}^k U_{\delta_i, \xi_i}(x)\quad \mbox{as}\,\, \e\to 0^+\eeq where $k\geq 3$ is an odd fixed integer, $\delta_i:=\delta_i(\e)\to 0$ as $\e\to 0^+$ and \beq\label{bubble} U_{\delta, \xi}(x):=\alpha_N \left(\frac{\delta}{\delta^2+|x-\xi|^2}\right)^{\frac{N-2}{2}},\quad \delta>0, \xi\in\mathbb R^N,\quad \alpha_N:=(N(N-2))^{\frac{N-2}{4}}\eeq
 are all the positive solutions of the limiting equation \beq\label{le}-\Delta U=U^{2^*-1},\quad \mbox{in}\,\,\ \mathbb R^N.\eeq In order to explain the major challenge in the construction, we first fix some notations.\\
 We let $G(x, y)$ be the Green function of $-\Delta$ over $\Omega$ under Dirichlet boundary conditions, namely $G$ satisfies 
 \beq\label{green}
 \left\{\begin{aligned}&-\Delta_x G(x, y)=\delta_y(x),\quad &x\in\Omega\\
 &G(x, y)=0, \quad &x\in\partial\Omega.\end{aligned}\right.\eeq
 Let also $H(x, y)$ be its regular part: 
 \beq\label{H}
H(x, y)=\frac{1}{\omega_{N-1}(N-2)}\frac{1}{|x-y|^{N-2}}-G(x, y),\eeq where $\omega_{N-1}$ is the surface measure of the unit sphere in $\mathbb R^N$. The value of $H$ on the diagonal, namely the function $\tau(x):=H(x, x)$, is called the Robin function of the domain $\Omega$.  Given $k\in \N$, $k\geq 2$, and $\boldsymbol\xi= (\xi_1,\ldots,\xi_k)\in \Omega^k$, we also denote by $M(\boldsymbol\xi)=(m_{ij})_{1\leq i, j\leq k}$ the matrix defined by
\beq\label{matrixM}
m_{ii}:=\tau(\xi_i),\quad m_{ij}=-G(\xi_i, \xi_j), \,\, i\neq j\,, \eeq and we denote by $\Lambda_1(\boldsymbol\xi)$ its first eigenvalue. Using the Perron-Frobenius theorem it is easy to show (see \cite{BLR}) that the eigenvalue $\Lambda_1(\boldsymbol\xi)$ is simple and the corresponding eigenvector can be chosen to
have strictly positive components. We denote by $\boldsymbol\ee(\boldsymbol\xi):=(\ee_1(\boldsymbol\xi), \ldots, \ee_k(\boldsymbol\xi))^T\in\mathbb R^k$ the unique vector such that
\begin{equation}\label{impM}M(\boldsymbol\xi)\boldsymbol\ee(\boldsymbol\xi)=\Lambda_1(\boldsymbol\xi)\boldsymbol\ee(\boldsymbol\xi)\ \hbox{and}\  \boldsymbol\ee_1(\boldsymbol\xi)=1.\end{equation}
\noindent After the reduction procedure, the major difficulty in \cite{Daprile} is to find the parameters $\delta_i\sim \e^{\frac{1}{N-2}}d_i$ and the points $\xi_i$ which turn out  to be {\it stable} critical points of the functional \beq\label{psi+} \Psi_+(\mathbf{d}, \boldsymbol\xi):=\frac 12 \left(M(\boldsymbol\xi) \mathbf{d}^{\frac{N-2}{2}}, \mathbf{d}^{\frac{N-2}{2}}\right) + B \log (d_1 \cdot\ldots\cdot d_k),\quad B>0\eeq on every compact subset of its domain $$\mathcal M:=\left\{(\mathbf{d}, \boldsymbol\xi)\,:\,  d_i>0,\,\, \xi_i\in\Omega\,\,\, \xi_i\neq \xi_j\,\, i\neq j\right\}.$$ Clearly, the dependence of the functional $\Psi_+$ on the scaling parameters and also on the points makes the situation more complicated. So in \cite{Daprile} a delicate min-max scheme is set up in a good configuration space in order to characterize a topologically nontrivial critical point of $\Psi_+$ which is stable under $C^1$ - perturbations. The geometric assumption on the small size of the hole is crucial in order to compare the Green function and its regular part near the hole. This helps to find the good configuration space that is confined in the
set of $k$-tuples $\boldsymbol\xi=(\xi_1, \ldots, \xi_k)$ for which the first eigenvalue $\Lambda_1(\boldsymbol\xi)$ is negative or, equivalently, the quadratic
form associated to the matrix $M(\boldsymbol\xi)$ is not positive definite. We remark that in \cite{DFM2}, when $\Omega$ is an annulus, the authors are able to show that when $k$ is large then the minimum of the first eigenvalue of the matrix $M(\boldsymbol\xi)$ is negative.  \\\\
If, instead, one considers the subcritical problem $(\mathcal P_-^\e)$, then a similar construction can be done in any kind of domain. However, if one wants to construct a positive solution with $k$ blow-up points as in \eqref{approx} then, after the reduction scheme, the parameters $d_i>0$ and the points $\xi_i$ are now the {\it stable} critical points of the functional   \beq\label{tildepsi} \Psi_-(\mathbf{ d}, \boldsymbol\xi):=\frac 12 \left(M(\boldsymbol\xi) \mathbf{d}^{\frac{N-2}{2}}, \mathbf{d}^{\frac{N-2}{2}}\right) -B\log (d_1 \cdot\ldots\cdot d_k), \quad B>0\eeq on every compact subset of its domain $\mathcal M$ defined as before.\\ In \cite{BLR} it is shown the existence of a positive solution  of $(\mathcal P_-^\e) $ with $k$ peaks in a generic bounded smooth domain. However, now, the good configuration space is confined in the set of $k-$ tuples $\boldsymbol\xi$ for which the first eigenvalue $\Lambda_1(\boldsymbol\xi)$ is positive. The case of two peaks solutions was previously presented in \cite{Rey3}. At the moment, as far as we know, for problem $(\mathcal P_-^\e)$ there are no results that guarantee the positivity of $\Lambda_1(\boldsymbol\xi)$ also in particular domains. \\\\
We also remark that problems $(\mathcal P_\pm^\e)$ have many analogies with the following problems of Brezis-Nirenberg form, i.e.
\[
(\mathcal {BN}_\mp)\quad \begin{cases}\begin{aligned} &-\Delta u =u^{2^*-1}\mp\e u, \quad &\mbox{in}\,\, \Omega\\
& u>0\quad &\mbox{in}\,\, \Omega\\
&u=0\quad &\mbox{on}\,\,\ \partial\Omega\end{aligned}
\end{cases}
\hspace{0.5cm}
\]
These problems are studied in \cite{BN}, where a positive solution for $(\mathcal{BN}_+^\e)$ is found provided $\e\in (0, \lambda_1(\Omega))$ (with $N\geq 4$), being $\lambda_1(\Omega)$ the first eigenvalue of $-\Delta$ with Dirichlet boundary conditions. A similar result holds in dimension $N=3$ but now $\e\in (\lambda^*(\Omega), \lambda_1(\Omega))$ with $\lambda^*(\Omega)>0$.\\ In the famous paper of Brezis and Nirenberg it is also showed that a Pohozaev's identity holds for problem $(\mathcal{BN}_-^\e)$ and again one has to play with the geometry and the topology of the domain in order to obtain a solution of the problem (see \cite{MuPi1} where it is considered the double blow-up phenomenon in a domain with a small hole in dimension $N\geq 5$ and references therein).
\\ The literature is very wide for problem $(\mathcal{BN}_+^\e)$ (see \cite{LVWW} for the history and the main results on it). When $N\geq 5$,  the question of the existence of a multi-peak solution is raised  in \cite{MuPi}, showing that the function responsible for the existence is 
\beq\label{phi} \Phi_+(\mathbf d, \boldsymbol\xi):=\frac 12 \left(M(\boldsymbol\xi) \mathbf{d}^{\frac{N-2}{2}}, \mathbf{d}^{\frac{N-2}{2}}\right) - \frac 12 B\sum_{j=1}^k d_j^{2}, \quad B>0.\eeq Again a {\it stable} critical point $(d_0, \boldsymbol\xi_0)\in\mathcal M$ of $\Phi_+$ gives the existence of a family of positive solutions that blows up at $\boldsymbol\xi_0$. Moreover they also exhibit classes of domains where such critical points of $\Phi_+$ can be found (f.i. $\Omega$ is a dumbbell type domain).\\
The multiple concentration for problem $(\mathcal{BN}_+^\e)$ in low dimensions, namely $N=3$ and $N=4$, is a very recent story (see \cite{MS} and \cite{PRV}) in which also the case of the annulus is considered as an example of domain in which a critical point of $\Phi_+$ can be found. However, to its difficulty, in \cite{MS} and \cite{PRV} only the case of two peaks in the annulus is considered.\\\\
In what follows we want to cover some gaps for these problems. Here we consider the case of the annulus $$\Omega_\rho:=\left\{x\in\mathbb R^N\,\,\:\,\,\,\ \rho<|x|<1\right\}, \quad 0<\rho<1,$$
aiming at showing the existence of multi-bubble solutions concentrating at $k$ points in the $(x_1, x_2)$-plane at equal distance from the origin and spaced at uniform angles, that is
\begin{equation}\label{punti}
	\xi_{j}(r) =(re^{2\pi \mathtt i \frac{j-1}{k}},\boldsymbol 0) \in \mathbb{R}^{2} \times \mathbb{R}^{N-2} \ \ \ \ j=1, \ldots k,\quad r\in (\rho, 1).
\end{equation}
We look for symmetric solutions in the space 
\begin{equation*}  \begin{aligned}
		H_s:=&\left\{u\in H^1_0(\Omega_\rho)\, :\, u\,  \hbox{is even in}\, x_h, \, h=2,...,N,\right.\\
		&\left.u(r\cos\theta, r\sin\theta, x'')=u\left(r\cos\left(\theta+\frac{2\pi}{k}\right),r\sin \left(\theta+\frac{2\pi}{k}\right),x''\right)\right\}.
	\end{aligned}
\end{equation*}

\noindent The existence of such solution depends on the inner radius $\rho$. For problem $(\mathcal P_-^\e)$ and $(\mathcal{BN}_+^\e)$ we prove that the size of $\rho$ increases (i.e. the annulus is very thin) when the number of peaks becomes larger and larger and no solution of this form exists when $\rho$ is near zero (i.e. the hole of the annulus is very small) .\\ On the contrary, for problems $(\mathcal P_+^\e)$ and $(\mathcal{BN}_-^\e)$ the situation is opposite, namely a solution of this form can be found if  the annulus has a small hole, while it does not exist when $\rho$ is near $1$.\\
We remark that for problem  $(\mathcal{BN}_-^\e)$ a double-blow up is done in \cite{MuPi1} for $N\geq 5$ and nothing is known in dimension $N=4$. Here we consider the general case for $N\geq 4$.\\\\
The main results of the paper are the following.
\begin{theorem}\label{main1}
Let $k\geq 2$  and $N\geq 3$ be fixed. Then there exists $\rho_k\in (0, 1)$  such that: 
\begin{itemize}
	\item[i)] if $\rho\in (\rho_k, 1)$ then there exists $r_0=r_0(\rho)\in (\rho,1)$ such that, for any sufficiently small $\e>0$, there exists a positive solution of $(\mathcal P_-^\e)$ and  $(\mathcal{BN}_+^\e)$ in $H_s$, which concentrates as $\e  \to 0$ at $k$ points $\xi_1(r_0),\ldots \xi_k(r_0)$ defined by \eqref{punti}.
	\item[ii)] for any $\rho \in (0,\rho_k)$, problems $(\mathcal P_-^\e)$ and $(\mathcal{BN}_+^\e)$ have no family of positive solutions in $H_s$  concentrating as $\e \to 0$ at $k$ points of the form \eqref{punti}.
\end{itemize}
\end{theorem}
\begin{theorem}\label{main2}
Let $k\geq 2$ be fixed. Then there exists $\rho_k\in (0, 1)$  such that:
\begin{itemize}
\item[i)]
 if $\rho\in (0,\rho_k)$, there exists $r_0=r_0(\rho)\in (\rho,1)$, such that for any small $\e$  the problem $(\mathcal P_+^\e)$ in dimension $N\geq 3$ and the problem $(\mathcal {BN}_-^\e)$ in dimension $N\geq 4$ have a positive solution in $H_s$ which concentrates as $\e \to 0 $ at $k$ points $\xi_1(r_0),\ldots \xi_k(r_0)$ defined by
\eqref{punti}. 
\item[ii)] if $\rho \in (\rho_k,1)$, problems $(\mathcal P_+^\e)$ and $(\mathcal{BN}_-^\e)$ have no family of positive solutions in $H_s$ concentrating as $\e\to 0$ at $k$ points of the form \eqref{punti}.  
\end{itemize}
\end{theorem}

\begin{remark}\label{remark1}
A precise characterization of $\rho_k$ is given in Section \ref{proof}. In particular, we prove that $\rho_k\to 1^-$ as $k\to +\infty$ (see Remark \ref{soglia}). This means that for problems $(\mathcal P_-^\e)$ and $(\mathcal{BN}_+^\e)$, when the number of peaks is very large, the solution of Theorem \ref{main1} exists only on thin annuli. On the contrary,  for problems $(\mathcal P_+^\e)$ and $(\mathcal{BN}_-^\e)$ the existence of such solution is guaranteed when the hole of the annulus is small. We stress that in both cases the concentration is inside the annulus and not on the boundary.
\end{remark}

\begin{remark}\label{remark2}
The result in Theorem \ref{main1} is known for problem $(\mathcal{BN}_+^\e)$ in dimension $N=3$ (see \cite{MS}). Here we use a different method in order to prove it and, in addition, we provide more precise quantitative estimates for $\rho_k$. \\ For problem $(\mathcal{BN}_+^\e)$ in dimension $N=4$, Theorem \ref{main1} together with Remark \ref{soglia} represents an answer to the conjecture given in \cite{PRV} (see (4.41) in \cite{PRV}).\\ Finally, Theorem \ref{main1} and Theorem \ref{main2} generalize some known results in literature. For instance, in \cite{MuPi1} only the double blow-up for problem $(\mathcal{BN}_-^\e)$ when $N\geq 5$ is considered in a general bounded domain with a small hole. Moreover, in \cite{Daprile} it is considered only the case with $k$ odd for problem $(\mathcal P_+^\e)$ for technical reasons. Here we do not have these restrictions. \end{remark}

As we will see in the next sections, the difficult part in order to prove Theorems \ref{main1} and \ref{main2} is to show that  the first eigenvalue of the matrix $M(\boldsymbol\xi)$ has the good sign. To do so we will make use of the explicit form of the Green function and of the Robin function on the annulus given in \cite{GV}. We will also exploit the properties of the zonal harmonics and of the Gegenbauer polynomials that are involved in the definition. \\\\ 
The paper  is organized in the following way: in Section \ref{preliminari} we will give some sketch of the reduction scheme and we discuss also the properties of the matrix $M(\boldsymbol\xi)$ in the annulus. In Section \ref{segno} we study the sign of the first eigenvalue of the matrix $M(\boldsymbol\xi)$. In Section \ref{proof} we prove Theorems \ref{main1} and \ref{main2} and we will discuss the properties of the minimum value of the first eigenvalue of the matrix $M(\boldsymbol\xi)$ (see Proposition \ref{main3} and Proposition \ref{unicozero}).

\section{Preliminaries}\label{preliminari}
\subsection{The reduction scheme}\label{riduzione}
In this section we outline the main steps of the reduction procedure in order to deduce the functions whose stable critical points are responsible for the existence of solutions of the problems $(\mathcal P_\pm^\e)$ and $(\mathcal{BN}_\pm^\e)$ respectively in dimension $N\geq3$ and $N\geq 5$.\\For $(\mathcal{BN}_+^\e)$ with  $N=3$ and $N=4$ the results for the annulus are known in the literature. We only outline them in Section \ref{cp}, where we also discuss $(\mathcal{BN}_-^\e)$ in dimension $N=4$ (see Remark \ref{n34}).\\

\noindent As mentioned in the introduction, for all $r\in (\rho,1)$, we will consider a symmetric configuration of points $\xi_1(r),\ldots \xi_k(r)$ of the form \eqref{punti}.
Let us also introduce $PU_{\delta,\xi}$ the projection onto $H^1_0(\Omega_\rho)$ of the bubble $U_{\delta, \xi}$ defined in \eqref{bubble}, namely the solution of the Dirichlet problem
\beq\label{proj}
\left\{\begin{aligned} &-\Delta  PU_{\delta,\xi}=U_{\delta,\xi}^{2^*-1}\quad &\mbox{in}\,\, \Omega_\rho\, \\
&PU_{\delta,\xi}=0\quad &\mbox{on}\,\, \partial\Omega_\rho.\end{aligned}\right.\eeq
We recall the following relation (see \cite{Rey2}, Lemma 2).
\begin{lemma}\label{HGU}
As $\d\to 0$ 
\beq\label{HU}
PU_{\d, \xi}(x)=U_{\d, \xi}(x)-\mathfrak C\d^{\frac{N-2}{2}}H(x, \xi)+\mathcal O(\d^{\frac{N+2}{2}}),\quad x\in\Omega_\rho, \eeq $C^1-$ uniformly on compact sets of $\Omega_\rho $ and
\beq\label{GU}
PU_{\d, \xi}(x)=\mathfrak C\d^{\frac{N-2}{2}}G(x, \xi)+\mathcal O(\d^{\frac{N+2}{2}}),\quad x\in\Omega_\rho, \eeq $C^1-$ uniformly on compact sets of $\bar\Omega_{\rho}\setminus\{\xi\}$ where $\mathfrak C=(N-2)\alpha_N \omega_{N-1}$.\end{lemma}

We look for a solution of the problems $(\mathcal P_\pm^\e)$ and $(\mathcal {BN}^\e_\pm)$ of the form
\beq\label{sol}
u_\e(x):=\sum_{j=1}^k PU_{\delta,\xi_j}+\phi_\e
\eeq
where the concentration parameters $\d:=\d(\e)\to 0$ as $\e\to 0$ are chosen in the following way
\beq\label{cp1}
\delta:=\left\{\begin{aligned} & \e^{\frac{1}{N-2}}d\quad &\mbox{for problems}\,\, (\mathcal P_\pm^\e)\,\, \mbox{and}\,\, N\geq 3\\
& \e^{\frac{1}{N-4}}d\quad &\mbox{for problems}\,\, (\mathcal{BN}_\pm^\e)\,\, \mbox{and}\,\, N\geq 5
\end{aligned}\right.\eeq	 
where $d\in (a, b)$ for some $0<a<b$ and $\phi_\e\in H_s$ is a remainder term chosen appropriately (see \eqref{kbot}). 
To carry out the construction of a solution of this type, we first introduce an intermediate problem as follows.\\ We let
\beq\label{nucleo}
Z_{j, 1}:=\frac{\partial P U_{\delta, \xi_j}}{\partial\delta},\quad Z_{j, 2}:=\frac{\partial PU_{\delta, \xi_j}}{\partial r}, \quad j=1, \ldots, k.\eeq
We also let $$\mathcal K_\e:={\rm span}\left\{ Z_{j, 1}, Z_{j, 2}\,:\, j=1, \ldots, k\right\}$$ and \beq\label{kbot}\mathcal K_\e^\bot:=\left\{\phi\in H_s\,\,:\,\,\, (\phi, Z_{j, i})_{H^1_0(\Omega_\rho)}=0\,\, j=1, \ldots, k,\,\, i=1, 2\right\}.\eeq As usual the first step consists in looking for $\phi_\e$ in $\mathcal K_\e^\bot$ by using a Banach contraction argument and then solving the reduced problem looking for $d$ and $r$.\\
In what follows we let $V_\e:=\sum_{j=1}^k PU_{\delta,\xi_j}$ and $$f_{\pm \e}(s):=\left\{\begin{aligned} & s^{2^*-1\pm \e}\quad&\mbox{for problems}\,\, (\mathcal P_\pm^\e)\\
&s^{2^*-1}\pm \e s \quad&\mbox{for problems}\,\, (\mathcal{BN}_\pm^\e).\end{aligned}\right.$$
By using the results in \cite{DFM1, BMP, Rey3, BLR, PRV, MuPi}, it is possible to prove the following result.
\begin{proposition}\label{esistenzaphi}
Let $0<a<b$. There exists $\e_0>0$ and a constant $C>0$ such that for each $\e\in(0, \e_0)$ and each $(d, r)\in (a, b)\times(\rho, 1)$ there exists a unique $\phi_\e\in\mathcal K_\e^\bot$ satisfying
$$\Delta (V_\e+\phi_\e)+f_{\pm\e}(V_\e+\phi_\e)\in\mathcal K_\e$$ and $$\|\phi_\e\|_{H^1_0(\Omega_\rho)}\lesssim \left\{\begin{aligned}&\e\quad&\mbox{for problems}\,\, (\mathcal{P}_\pm^\e)\\
&\e^{\frac{N+2}{2(N-4)}}\quad&\mbox{for problems}\,\, (\mathcal{BN}_\pm^\e)\,\,\mbox{and}\,\, N\geq 7\\
&\e^{2}|\ln\e|\quad&\mbox{for problems}\,\, (\mathcal{BN}_\pm^\e)\,\,\mbox{and}\,\, N= 6\\
&\e^{\frac{5}{2}}\quad&\mbox{for problems}\,\, (\mathcal{BN}_\pm^\e)\,\,\mbox{and}\,\, N=5.\\\end{aligned}\right. $$ Moreover the map $(d, r)\in (a, b)\times (\rho, 1) \mapsto \phi_\e :=\phi_\e(d, r)\in H^1_0(\Omega_\rho)$ is $C^1$.
\end{proposition}
\noindent Now, let us consider the following energy functional associated to problems $(\mathcal P_\pm^\e)$ and $(\mathcal{BN}_\pm^\e)$ 
\beq\label{funz}
J_\e(u):=\frac 12 \int_{\Omega_\rho} |\nabla u|^2\, dx -\int_{\Omega_\rho} F_{\pm \e}(u)\, dx, \quad u\in H_s, \eeq
where $F_{\pm \e}(s):=\int_0^s f_{\pm \e}(t)\, dt$.\\ Solutions of problems $(\mathcal P_\pm^\e)$ and $(\mathcal{BN}_\pm^\e)$  are critical points of $J_\e$.\\Finally, we introduce the reduced functional
\begin{equation}\label{JFrak}
\mathfrak J_\e (d, r):= J_\e (V_\e+\phi_\e),\quad (d, r)\in (a, b)\times (\rho, 1)
\end{equation}
where $\phi_\e$ is the function of Proposition \ref{esistenzaphi}. \\ Then it is possible to show the following Lemma (see \cite{DFM2}).
\begin{lemma}\label{lemma1}
The pair $(d, r)\in (a, b)\times (\rho, 1)$ is a critical point of $\mathfrak J_\e$ if and only if the corresponding function $u_\e=V_\e+\phi_\e$ is a critical point of $J_\e$.\end{lemma}
Finally we describe the expansion for $\mathfrak J_\e$ which can be obtained as in \cite{DFM2, DFM3, Daprile, BMP, PRV, MuPi}. For simplicity we let 
\beq\label{lambdar}
\Lambda(r):= \sum_{j=1}^k \tau_\rho(\xi_j(r))-\sum_{i\neq j} G_\rho(\xi_i(r), \xi_j(r)).\eeq
\begin{proposition}\label{exp}
As $\e\to 0$, the following asymptotic expansion holds for problems $(\mathcal P_\pm^\e)$ for $N\geq 3$
\beq\label{ppm}
\mathfrak J_\e(d, r)=\mathfrak c_0+ \e \left(\mathfrak c_1 d^{N-2}\Lambda(r)\pm\mathfrak c_2 \log d\right)+o\left(\e\right)
\eeq
while, for problems $(\mathcal{BN}_\pm^\e)$ in dimensions $N\geq 5$ it holds
\beq\label{bnpm5}
\mathfrak J_\e(d, r)= \mathfrak d_0 +\e^{{\frac{N-2}{N-4}}}\left(\mathfrak d_1 d^{N-2}\Lambda(r)\mp\mathfrak d_2 d^2\right)+o\left(\e^{{\frac{N-2}{N-4}}}\right)
\eeq
where $\mathfrak c_i$ and $\mathfrak d_i$ in \eqref{ppm} and \eqref{bnpm5} are some known positive constants. The previous estimates hold $C^1-$ uniformly with respect to $(d, r)\in (a, b)\times (\rho, 1)$.
\end{proposition}
%
%
Thus in order to construct a solution of problems $(\mathcal P_\pm^\e)$ and $(\mathcal{BN}_\pm^\e)$ such as the one predicted in Theorems \ref{main1}, \ref{main2} it remains
to find a critical point of $\mathfrak J_\e$. This will be accomplished in the Section \ref{cp}.\\\\

\subsection{About the function $\Lambda(r)$}\label{lambdarp}
First we recall some known results about the Green function and the Robin function on the annulus. Indeed, their  explicit expression is given in \cite{GV}, namely
\begin{equation*}
\tau_\rho(x):=\frac{1}{\omega_{N-1}}\sum_{m=0}^{\infty}d_m Q_m(|x|),
\end{equation*}
where, for all $m \geq 1$,  
\begin{equation*}
d_{m} = \binom{m+N-2}{N-2} + \binom{m+N-3}{N-2}
\end{equation*}
and
\begin{equation}\label{Qm}
Q_m(|x|):=\frac{\rho^{2m+N-2} - 2\rho^{2m+N-2}|x|^{2m+N-2} + |x|^{4m+2N-4}}{(2m+N-2)|x|^{2N+2m-4}(1-\rho^{2m+N-2})}.
\end{equation}
Instead, 
\begin{equation*}
G_\rho(x, y):=\frac{1}{\omega_{N-1}} \Bigg[ \frac{1}{(N-2)|x-y|^{N-2}} - \sum_{m=0}^{\infty} Q_m(x, y) Z_m\left(\frac{x}{|x|}, \frac{y}{|y|}\right) \Bigg],
\end{equation*}
where
\begin{equation*}
Q_m(x, y):=\frac{\rho^{2m+N-2} - \rho^{2m+N-2}(|x|^{2m+N-2}+|y|^{2m+N-2}) + |x|^{2m+N-2}|y|^{2m+N-2}}{(2m+N-2)(|x||y|)^{m+N-2}(1-\rho^{2m+N-2})},
\end{equation*}
and $Z_m(\zeta, \eta)$ represents the zonal harmonics of degree $m$ which have a particularly simple expression in terms of the Gegenbauer (or ultra-spherical) polynomials $C_{m}^{\lambda}$. The latter can be defined in terms of generating functions. If we write (see \cite{SW} p. 148)
\begin{equation*}
(1-2rt+r^{2})^{-\lambda} = \sum_{m=0}^{\infty} C_m^{\lambda}(t)r^{m},
\end{equation*}
where $0 \leq r<1$, $|t| \leq 1$ and $\lambda>0$, then the coefficient $C_m^{\lambda}$ is called Gegenbauer polynomial of degree $m$ associated with $\lambda$. \ \\
In the following we let, for all $m \geq 1$
\begin{equation}\label{binom}
d_{m} = A_{N,m} \cdot c_{N,m}\,,
\end{equation}
where 
\begin{equation}\label{ANm}
A_{N,m} : =\binom{N+m-3}{N-3}, \quad c_{N,m} = \frac{N+2m-2}{N-2}.
\end{equation}
An important result is to understand the link between the zonal harmonics and the Gegenbauer polynomials. Indeed, the following result is proved in \cite{GV}.
\begin{theorem}[\cite{GV}, Theorem $2.1$]\label{zmpm}
If $N>2$ and $m=0, 1, 2, \ldots$ then we have that for all $x',y'$ such that $|x'|=|y'|=1$ it holds
\begin{equation*}
Z_{m}(x',y')= c_{N,m} C_{m}^{\frac{N-2}{2}}(x' \cdot y').
\end{equation*}
\end{theorem}
Now, by using the expression and the properties of the Green function and 
of the Robin function in the annulus $\Omega_\rho$ we see that
$$\tau_\rho(\xi_j(r))=\tau_\rho(\xi_1(r))$$ since 
$|\xi_j(r)|=r$ for all $j$. \\ Moreover, we have also that $G_\rho(\xi_{l}(r), \xi_{j}(r)) = G_\rho(\xi_{l+1}(r), \xi_{j+1}(r))$.\\  Hence $$\Lambda(r)=k\left(\tau_\rho(\xi_1(r))-\sum_{j=1}^{k-1}G_\rho(\xi_1(r), \xi_{j+1}(r))\right)=:k \Lambda_1(r). $$
$\Lambda_1(r)$ is nothing else that the least eigenvalue of the matrix $M(\boldsymbol \xi)=\mathtt M(r)$. \\Indeed, first of all we see that the matrix $\mathtt M(r)$ is a circulant matrix, i.e. each column is obtained from the previous one by a rotation in the components:
\begin{center}
$A= \left( \begin{matrix}
a_0 & a_{k-1} & a_{k-2} & \cdots & a_2 & a_1 \\
a_1 & a_0 & a_{k-1} & \cdots & a_3 & a_2 \\
a_2 & a_1 & a_0 & \cdots & a_4 & a_3 \\
\vdots & \vdots & \vdots & & \vdots & \vdots \\
a_{k-1} & a_{k-2} & a_{k-3} & \cdots & a_1 & a_0 
\end{matrix} \right)$
\end{center}
with $$a_0 := \tau_\rho(\xi_{1}(r)),\quad\hbox{ and}\quad 
a_{j} := - G_\rho(\xi_{1}(r), \xi_{j+1}(r)),\,\ j=1, \cdots k-1.
$$
It is well known that the eigenvalues of a circulant matrix are explicitly given by  
\begin{equation*}
 \Lambda_{\ell}(r) = \sum_{j=0}^{k-1} a_j e^{\frac{2\pi \mathtt i}{k} j(\ell-1)}, \ \ell=1, \cdots, k.
\end{equation*}
Moreover, it easily follows that $$a_{k-j}=-G_\rho(\xi_{1}(r), \xi_{k-j+1}(r))=-G_\rho(\xi_{1}(r), \xi_{j+1}(r))=a_j,$$
since $\xi_{j+1}(r)=r\left(\cos\frac{2\pi j}{k}, \sin\frac{2\pi j }{k}, \boldsymbol 0\right)$ and $$\begin{aligned}\xi_{k-j+1}(r)&=r\left(\cos\frac{2\pi (k-j)}{k}, \sin\frac{2\pi (k-j)}{k}, \boldsymbol 0 \right)=r\left(\cos\frac{2\pi j}{k}, -\sin\frac{2\pi j}{k}, \boldsymbol 0 \right)\end{aligned}$$  have equal scalar  product with $\xi_1(r)=r(1,0, \boldsymbol 0)$.\\ Hence, the matrix $\mathtt{M}(r)$ is also symmetric and then all the eigenvalues are real. \\ We claim that
$$\Lambda_1(r)=\tau_\rho(\xi_1(r))-\sum_{j=1}^{k-1}G_\rho(\xi_1(r), \xi_{j+1}(r)) $$
is simple.
Indeed
$$\begin{aligned} \Lambda_\ell(r)&=a_0 +\sum_{j=1}^{k-1}a_j e^{\frac{2\pi\mathtt i}{k}j(\ell-1)}=\tau_\rho(\xi_1(r))-\sum_{j=1}^{k-1}G_\rho(\xi_1(r), \xi_{j+1}(r))\mathtt Re\left(e^{\frac{2\pi\mathtt i}{k}j(\ell-1)}\right)\\
&>\tau_\rho(\xi_1(r))-\sum_{j=1}^{k-1}G_\rho(\xi_1(r), \xi_{j+1}(r))=\Lambda_1(r).
\end{aligned}$$

\subsection{Critical points of the reduced problem}\label{cp}
The finite-dimensional reduction carried out in Section \ref{riduzione} implies that our problem reduces to investigate
the existence of critical points of the functional $\mathfrak J_\e$ in \eqref{JFrak}, which is a small $C^1$-perturbation of the function $\Psi$ defined in the following way:

\beq\label{psi}
\Psi(d, r):=\left\{\begin{aligned}&\mathfrak c_1 d^{N-2}\Lambda_1(r)\pm\mathfrak c_2 \log d\quad &\mbox{for}\,\, (\mathcal P_\pm^\e)\,\, \mbox{and}\,\, N\geq 3\\
&\mathfrak d_1 d^{N-2}\Lambda_1(r)\mp\mathfrak d_2 d^2\quad &\mbox{for}\,\, (\mathcal{BN}_\pm^\e)\,\, \mbox{and}\,\, N\geq 5\\\end{aligned}\right.
\eeq
Hence we are interested in the existence of stable critical points of $\Psi$ in the sense of the following definition introduced by Y.Y. Li in \cite{yyl}. 
\begin{definition}\label{yy1}
Given a smooth function $f:D\subset\mathbb R^n\to\mathbb R,$ a set $\mathscr K$ of critical points of $f$ is    stable if there exists a neighbourhood $\Theta$ of $\mathscr K$ such that  $\mathtt{deg}(\nabla f,\Theta,0)\not=0,$ where 
$\mathtt{deg}$ denotes  the Brouwer degree. \end{definition}
Examples of stable critical sets are listed below.
\begin{itemize}
\item  $\mathscr K$ is a strict local minimum (or maximum) set of $f,$ i.e. $f(x)=f(y)$ for any $x,y\in \mathscr K$ and 
$f(x)< f(y)$ (or $f(x)> f(y)$) for any $x\in \mathscr K$ and $y\in \Theta\setminus\mathscr K$.
\item $\mathscr K=\{x_0\}$ is an isolated critical point of $f$ with $\mathtt{deg}(\nabla f,B(x_0,\rho),0)\not=0$  for some small $\rho>0$ (e.g. $x_0$ is a non-degenerate critical point of $f$).\\
\end{itemize}
 
\noindent It is immediate to check that for any $\rho>0$ there exists a strict minimum point $r_0=r_0(\rho)\in (\rho,1)$ of the function $\Lambda_1,$
since
$$\lim\limits_{r\to \rho}\Lambda_1(r)=\lim\limits_{r\to 1}\Lambda_1(r)=+\infty.$$
At this point it is important to establish the sign of $\displaystyle{\min_{r\in(\rho, 1)}\Lambda_1(r)}$ in order to deduce the existence of a stable critical point of $\Psi$.\\ In what follows we remark the situation in the remaining cases (i.e. $N=3$ and $N=4$).\\


\begin{remark}\label{n34}
In dimensions $N=3, 4$ the reduction scheme is more involved with respect to the dimensions $N\geq 5$ for the problem $(\mathcal{BN}_+^\e)$.  However, in \cite{PRV}, the problem $(\mathcal{BN}_+^\e)$ is considered in dimension $N=4$ showing that a solution of the form \eqref{sol} exists and it concentrates at a stable critical point $r_0\in (\rho, 1)$ of $\Lambda_1(r)$  in which $\Lambda_1(r_0)>0$ with $\delta \sim e^{-\frac{\Lambda_1(r_0)}{\e}}$ (see Theorem 1.2 of \cite{PRV}). \\ The scheme is very similar for problem $(\mathcal{BN}_-^\e)$ in dimension $N=4$ (even if nothing is done in the literature) and hence, in this case, there exists a solution of the problem $(\mathcal{BN}_-^\e)$ if again $\Lambda_1(r)$ has stable critical points but now at negative level and the rate of concentration is $\delta \sim e^{\frac{\Lambda_1(r_0)}{\e}}$.\\\\ The case $N=3$ for the Brezis-Nirenberg problem is more delicate. Indeed, in \cite{BN} it is shown that a positive solution exists if $\e\in (\lambda^*, \lambda_1)$ where $\lambda^*$ is positive and it depends on the domain while $\lambda_1$ is the first eigenvalue of $-\Delta$ with Dirichlet boundary condition. Hence in dimension $N=3$, we can construct a solution of the form \eqref{sol} when $\e$ goes to some $\lambda^*>0$. This value was also characterized by Druet in \cite{Druet} as $$\lambda^*:=\sup\left\{\e>0\,:\, \min_\Omega \tau_\e>0\right\},\,$$ where now $\tau_\e$ is the Robin function defined as the trace of the regular part of the Green function $G_\e$ of the operator $-\Delta-\e {\rm Id}$.\\
Besides, it is also shown in \cite{Druet}, that least energy solutions $u_\e$ for $\e\downarrow\lambda^*$
constitute a single bubble with blowing-up near the set where $\tau_{\lambda^*}$ attains its minimum value zero.\\
In \cite{DDM}, a solution with a single positive bubble that concentrates around critical points of $\tau_{\lambda^*}$ at zero level is constructed, while the multi-bubble case was recently studied by Musso and Salazar in \cite{MS}. Here the matrix is  
$M_\e(\boldsymbol\xi)=(m^\e_{ij})_{1\leq i, j\leq k}$ defined by
\beq\label{matrixMe}
m^\e_{ii}:=\tau_\e(\xi_i),\quad m_{ij}=-G_\e(\xi_i, \xi_j), \,\, i\neq j\eeq 
and they are able to show the existence of a multispike positive solution that concentrates at a critical point at level zero of the function $\psi(\boldsymbol\xi)={\rm det} M_{\lambda^*}(\boldsymbol\xi)$.\\
In \cite{MS} it is also considered the case of the annulus $\Omega_\rho$ showing that for each $k\geq 2$ there exists $0 < \rho_k < 1$ such that if $\rho\in (\rho_k, 1)$,
then problem $(\mathcal{BN}_+^\e)$  in $\Omega_\rho$ has a solution with $k$ bubbles centered at
the vertices of a planar regular polygon for $\e$ that goes to some $\lambda^*\in (0, \lambda_1)$. 
The crucial point in their proof is to show that the first eigenvalue of the matrix $M_\e(\boldsymbol\xi)$ is strictly positive for $\e=0$ and then they use some continuity property to deduce their result. In this case the first eigenvalue for $\e=0$ is again given by $\Lambda_1(r)$ (see Section 7 of \cite{MS}).\\ We remark that the proof of the positivity of $\Lambda_1(r)$ is different from our argument.
\end{remark}

\section{The sign of $\Lambda_1(r)$}\label{segno}
From the previous analysis, it is clear that the main problem is to understand the sign of $\Lambda_1(r)$. So we write it explicitly by using the expressions of the Green function and of the Robin function in the annulus.\\
Firstly we point out that
\begin{equation}\label{xi}
\begin{aligned}
|\xi_1(r) - \xi_{j+1}(r)|^{N-2} & = \left( |\xi_1(r) - \xi_{j+1}(r)|^{2} \right)^{\frac{N-2}{2}} = \left( 2r^2 - 2 \xi_1(r) \cdot \xi_{j+1}(r) \right)^{\frac{N-2}{2}} 
\\& = \Bigg( 2r^2 - 2r^2 \cos \left( \frac{2\pi j}{k} \right) \Bigg)^{\frac{N-2}{2}} = r^{N-2} \Bigg(2\sin \left( \frac{\pi j}{k} \right)\Bigg)^{N-2}.
\end{aligned}
\end{equation}
for all $j=1, \cdots, k-1$. \ \\
Hence
\begin{equation*}
\sum_{j=1}^{k-1} \frac{1}{|\xi_1(r)-\xi_{j+1}(r)|^{N-2}} = \frac{\mathfrak C_{k,N}}{r^{N-2}},
\end{equation*}
where 
\begin{equation*}
\mathfrak C_{k,N} : =\sum_{j=1}^{k-1} \frac{1}{\left( 2 \sin \left( \frac{\pi j}{k} \right) \right)^{N-2}} .
\end{equation*} 

So, we have that
 \begin{equation*}
\begin{aligned}
\Lambda_{1}(r) & = \frac{1}{\omega_{N-1}} \Bigg[ \sum_{m=0}^{+\infty} d_{m} Q_{m}(r) - \sum_{j=1}^{k-1} \frac{1}{N-2} \frac{1}{|\xi_1-\xi_{j+1}|^{N-2}} + \sum_{j=1}^{k-1} \sum_{m=0}^{+\infty} Q_{m}(r) Z_{m} \left( \frac{\xi_1}{|\xi_1|}, \frac{\xi_{j+1}}{|\xi_{j+1}|} \right) \Bigg]\\
&=\frac{1}{\omega_{N-1}} \Bigg[- \frac{\mathfrak C_{k,N} }{N-2} \frac{1}{r^{N-2}}   + \sum_{m=0}^{+\infty} Q_m(r)\left(d_m+\sum_{j=1}^{k-1}Z_{m} \left( \frac{\xi_1}{|\xi_1|}, \frac{\xi_{j+1}}{|\xi_{j+1}|} \right)\right)\Bigg].
\end{aligned}
\end{equation*}
We see that
$$\Lambda_1(r)=\frac{1}{r^{N-2}}f(r)$$
where 
\begin{equation}\label{fr}f(r):={\frac{1}{\omega_{N-1}}}\Bigg[- \frac{\mathfrak C_{k,N} }{N-2}  + \sum_{m=0}^{+\infty} r^{N-2}Q_m(r)\left(d_m+\sum_{j=1}^{k-1}Z_{m} \left( \frac{\xi_1}{|\xi_1|}, \frac{\xi_{j+1}}{|\xi_{j+1}|} \right)\right)\Bigg].\end{equation} It is easy to see that 
\begin{equation}\label{iii}
\begin{aligned}
&(i)\quad &\displaystyle\min_{r\in(\rho, 1)}\Lambda_1(r)>0\quad &\iff\quad \displaystyle\min_{r\in(\rho, 1)}f(r)>0;\\
&(ii)\quad &\displaystyle\min_{r\in(\rho, 1)}\Lambda_1(r)=0\quad &\iff\quad \displaystyle\min_{r\in(\rho, 1)}f(r)=0;\\
&(iii)\quad & \displaystyle\min_{r\in(\rho, 1)}\Lambda_1(r)<0\quad &\iff\quad \displaystyle\min_{r\in(\rho, 1)}f(r)<0.\end{aligned}\end{equation}
Therefore, in order to study the sign of $\Lambda_1(r)$ a crucial role is played by $f(r)$.\\ 
In some cases, we need a different form of $f(r)$. Indeed, by using \eqref{binom} and Theorem \ref{zmpm}, we obtain that

\begin{equation}\label{fr1}
\begin{aligned}
f(r) & = \frac{1}{\omega_{N-1}} \Bigg[ - \frac{\mathfrak C_{k,N} }{N-2}   + \sum_{m=0}^{+\infty} r^{N-2}Q_{m}(r) c_{N,m} \Bigg( A_{N,m} + \sum_{j=1}^{k-1} C_{m}^{\frac{N-2}{2}} \left( \cos \left( \frac{2 \pi j}{k} \right) \right) \Bigg) \Bigg]\\
&=\frac{1}{\omega_{N-1}} \Bigg[- \frac{\mathfrak C_{k,N} }{N-2} + \sum_{m=0}^{+\infty} r^{N-2}Q_{m}(r) c_{N,m} \alpha_{k, m, N}\Bigg],
\end{aligned}
\end{equation}
where
\begin{equation}\label{alphakmN}
\alpha_{k, m, N}:= A_{N,m} + S_{k,m, N},
\end{equation}
with
\begin{equation}\label{Skm}
S_{k,m, N} := \sum_{j=1}^{k-1} C_{m}^{\frac{N-2}{2}} \left( \cos \left( \frac{2\pi j}{k} \right) \right).
\end{equation}


\subsection{Positivity of $\Lambda_1(r)$}
Here we will show the following:\\\\
{\it For any integer $k\geq 2$ there exists $\rho_k\in (0, 1)$ such that for any $\rho\in (\rho_k, 1)$
it holds true that }
\begin{equation}\label{condmin}
\min_{r\in (\rho, 1)} \Lambda_1(r)>0. \end{equation} Hence, from (i) of \eqref{iii}, this is equivalent to show that $\displaystyle\min_{r\in(\rho_k, 1)}f(r)>0$. In what follows, let us consider $f(r)$ in the form \eqref{fr1}. \\ Since we have to use different estimates for $\alpha_{k, m, N}$ in dimension $N=3$, $N=4$ and $N\geq 5$ we need to consider separately the various cases.
\subsubsection{The case $N=3$}\label{N3subsec}
In order to estimate $\alpha_{k, m, 3}$ we need to use the following expression of Gegenbauer polynomials (see $8.934, 2$ of \cite{GR})
\begin{equation*}
C_{m}^{\frac 12}(\cos \phi) = \sum_{i,l=0 \atop i+l=m}^{m} \frac{\Gamma(\frac{1}{2}+i) \Gamma(\frac{1}{2} + l)}{i! l! \pi} \cos ((i-l) \phi).
\end{equation*}
Moreover, it is well-known that, in this case, there is the following relation between Gegenbauer polynomials and Legendre polynomials $P_m(x)$ (see  $8.936, 3$ of \cite{GR}) $$C_m^{\frac 12}(x)=P_m(x)$$ and it is known that $P_m(1)=1$ (see  $8.828, 1$ of \cite{GR}). Then

\begin{equation}\label{cm121}
1 = P_m(1)=C_{m}^{\frac 12}(1)  =\sum_{i,l=0 \atop i+l=m}^{m} \frac{\Gamma(\frac{1}{2} + i) \Gamma(\frac{1}{2} + l)}{i! l! \pi}.
\end{equation}
So, in this case,
\begin{equation}\label{Skm3}
S_{k,m, 3} = \sum_{j=1}^{k-1} \sum_{i,l=0 \atop i+l=m}^{m} \frac{\Gamma(\frac{1}{2}+i) \Gamma(\frac{1}{2} + l)}{i! l! \pi}  \cos \left((i-l) \frac{2\pi j}{k}\right).
\end{equation}
By using \eqref{ANm}, \eqref{Skm3} and \eqref{cm121}, we have that
\begin{equation}\label{alphakm3}
\begin{aligned}\alpha_{k, m, 3}& = 1+ S_{k,m, 3}=\sum_{i,l=0 \atop i+l=m}^{m} \frac{\Gamma(\frac{1}{2}+i) \Gamma(\frac{1}{2} + l)}{i! l! \pi} \left(1 + \sum_{j=1}^{k-1} \cos \left((i-l) \frac{2\pi j}{k}\right) \right)\\
&=\sum_{l=0}^{m} \frac{\Gamma(\frac{1}{2}+m-l) \Gamma(\frac{1}{2} + l)}{(m-l)! l! \pi} \left(1 + \sum_{j=1}^{k-1} \cos \left((m-2l) \frac{2\pi j}{k}\right) \right).\\
\\\end{aligned}
\end{equation} 
We first study $ \displaystyle{\sum_{j=1}^{k-1} \cos \left(p \frac{2\pi j}{k}\right)}$ for some integer $p\in\mathbb N$. We need to distinguish two different cases, which are the following:
\begin{itemize}
\item[1.] $p$ is a multiple of $k$;
\item[2.] $p$ is not a multiple of $k$.
\end{itemize}
In the first case, we let $p=p_0 \cdot k, p_0 \in \mathbb{N}, p_0 \geq 1$. So
\begin{equation*}
\cos \left( p \cdot \frac{2\pi j}{k} \right) = \cos (p_0 \cdot 2\pi j) = 1,
\end{equation*}
from which we have that 
\begin{equation*}
\sum_{j=1}^{k-1} \cos \left( p\frac{2\pi j}{k} \right)=k-1.
\end{equation*}
In the second case we can observe that 
$
\sin \left(\frac{p\pi}{k} \right) \neq 0.
$
Hence
\begin{equation*}
\begin{aligned}
\sum_{j=1}^{k-1} \cos \left( p\frac{2\pi j}{k}  \right) &  = \frac{1}{ \sin\left(\frac{p\pi }{k}\right)} \sum_{j=1}^{k-1} \cos \left(  p\frac{2\pi j}{k} \right) \sin\left(\frac{p\pi }{k}\right) 
\\& = \frac{1}{2\sin\left(\frac{p\pi }{k}\right)} \sum_{j=1}^{k-1} \left[\sin \left(\frac{2\pi p}{k} j + \frac{\pi p}{k} \right)- \sin \left(\frac{2\pi p}{k} j - \frac{\pi p}{k} \right)\right]
\\& = \frac{1}{2\sin\left(\frac{p\pi }{k}\right)} \sum_{j=1}^{k-1} \left[\sin \left(\frac{2\pi p}{k} \left(j+\frac{1}{2}\right) \right) - \sin \left(\frac{2\pi p}{k} \left(j - \frac{1}{2}\right) \right) \right]
\\& = \frac{1}{2\sin\left(\frac{p\pi }{k}\right)} \Bigg[\sin \left(\frac{2\pi p}{k} \left(k-\frac{1}{2}\right) \right) - \sin \left(\frac{\pi p}{k} \right) \Bigg] 
\\& = \frac{1}{2\sin\left(\frac{p\pi }{k}\right)} \left( -2 \sin\left(\frac{p\pi }{k}\right) \right) = - 1.
\end{aligned}
\end{equation*}
At the end
\begin{equation}\label{sommatoria}
\sum_{j=1}^{k-1} \cos \left( p\frac{2\pi j}{k}  \right)=\left\{\begin{aligned} &k-1\quad &\mbox{if $p$ is a multiple of $k$}\\
&-1 \quad  &\mbox{if $p$ is not a multiple of $k$}.\end{aligned}\right.\end{equation}
Now, by using \eqref{alphakm3} we get that 
\begin{equation}\label{alpha3pari}\begin{aligned}\alpha_{k, 2m, 3}&=  \frac{\left(\Gamma\left(\frac 12 +m\right)\right)^2}{(m!)^2\pi}k +\sum_{l=0\atop l\neq m}^{2m} \frac{\Gamma(\frac{1}{2}+2m-l)\Gamma\left(\frac 12 +l\right) }{(2m-l)!l!  \pi} \left(1 + \sum_{j=1}^{k-1} \cos \left(2(m-l) \frac{2\pi j}{k}\right) \right)\\ 
&\geq  \frac{\left(\Gamma\left(\frac 12 +m\right)\right)^2}{(m!)^2\pi}k\geq \frac{k}{4m}\\
\end{aligned}\end{equation} since \eqref{sommatoria} holds and, by induction, it is easy to verify that 
\begin{equation*}
\frac{\Gamma(\frac{1}{2} + m)}{(m)!} \geq \frac{\sqrt\pi}{2\sqrt{m}},\quad \mbox{for any}\,\, m\geq 1.
\end{equation*}
Instead, by using again \eqref{sommatoria} we have that
\begin{equation}\label{alpha3dispari}
\alpha_{k, 2m+1, 3} \geq 0. 
\end{equation}


Now, we minimize $r Q_{m}(r)$, where (recalling \eqref{Qm})
\begin{equation*}
\begin{aligned}
r Q_{m}(r)  & = \frac{1}{(2m+1)(1-\rho^{2m+1})} \underbrace{\Bigg[ \frac{\rho^{2m+1}}{r^{2m+1}} -2\rho^{2m+1} + r^{2m+1} \Bigg]}_{:=g_{m,3}(r)}.
\end{aligned}
\end{equation*}

we have that 
\begin{equation*}\begin{aligned}
g'_{m,3}(r) &= -\frac{(2m+1) \rho^{2m+1}}{r^{2m+2}} + (2m+1) r^{2m} = \frac{2m+1}{r^{2m+2}} \Bigg[ -\rho^{2m+1} + r^{4m+2} \Bigg].\end{aligned}
\end{equation*}
So, we have a minimum point for $g_{m,3}(r)$ in $r=\sqrt{\rho}$. \ \\
In particular 
\begin{equation}\label{min3}
\min_{r\in(\rho,1)}r Q_m(r) = \sqrt{\rho}Q_{m}(\sqrt{\rho}) =\frac{2\rho^{m+\frac 12}}{(2m+1)(1-\rho^{2m+1})}\left(1-\rho^{m+\frac 12}\right).
\end{equation}
Finally, by using \eqref{min3}, \eqref{alpha3pari} and \eqref{alpha3dispari} we get
\begin{equation}\label{n3fin}
\begin{aligned}
 \min_{r \in (\rho,1)} f(r)  &= \frac{1}{\omega_2}\left(\sum_{m=0}^{+\infty} (2m+1) \alpha_{k,m, 3} \sqrt\rho Q_{m}(\sqrt\rho) -\mathfrak C_{k, 3}\right)
\\& \geq \frac{1}{\omega_2}\left(\sum_{m=0}^{+\infty} (4m+1)\alpha_{k, 2m, 3} \sqrt\rho Q_{2m}(\sqrt\rho)-\mathfrak C_{k, 3}\right)  \\
& \geq \frac{1}{\omega_2}\left(\frac{k}{2} \sum_{m=1}^{+\infty} \frac{1}{m}\frac{\rho^{\frac{4m+1}{2}}}{(1-\rho^{4m+1})} (1-\rho^{\frac{4m+1}{2}}) -\mathfrak C_{k, 3}\right)
\\& = \frac{1}{\omega_2}\left(\frac{k}{2}  \rho^{\frac{1}{2}} \sum_{m=1}^{+\infty}\frac{1}{m} \frac{\rho^{2m}}{1+\rho^{\frac{4m+1}{2}}} -\mathfrak C_{k, 3}\right)
\\& \geq  \frac{1}{\omega_2}\left(\frac{k}{4}  \rho^{\frac{1}{2}} \sum_{m=1}^{+\infty}\frac{\rho^{2m}}{m}-\mathfrak C_{k, 3}\right) \\
&=\frac{1}{\omega_2}\left( -  k \frac{\rho^{\frac{1}{2}}}{4} \log(1- \rho^{2}) -\mathfrak C_{k, 3}\right)>0,
\end{aligned}
\end{equation}
if $\rho$ is close to $1$, since $$-   \frac{k\rho^{\frac{1}{2}}}{4} \log(1- \rho^{2}) \xrightarrow{\rho \to 1^{-}} + \infty$$
and the claim holds. 
\subsubsection{The case $N=4$} \label{N4subsec}
As in the case $N=3$ we need to estimate $\alpha_{k, m, 4}$. Here $\mathfrak C_{k, 4}$ is explicitly given by (see, for instance \cite{AZ} or \cite{F} (Identity $23$)).
\begin{equation}\label{Ck4}
\mathfrak C_{k, 4} : = \frac 1 4\sum_{j=1}^{k-1} \frac{1}{\sin^2 \left( \frac{\pi j}{k} \right)}= \frac{k^2 -1}{12}.
\end{equation}
Moreover $c_{4, m}=m+1$ and $\alpha_{k, m, 4}$ is defined as in \eqref{alphakmN} with  $A_{4, m}:=m+1$ and 
\begin{equation*}
S_{k,m, 4} := \sum_{j=1}^{k-1} C_{m}^{1} \left( \cos \left( \frac{2\pi j}{k} \right) \right).
\end{equation*}
In this case, we use the relation between Gegenbauer's polynomials and Chebyshev polynomials by (see $8.937, 1$ of \cite{GR}), namely
\begin{equation}\label{cheby}
C_{m}^{1} \left( \cos t \right) = \frac{ \sin \left( (m+1) t \right)}{\sin  t }.
\end{equation}
Now, in order to evaluate $\alpha_{k,m, 4}$, we want to use the following recursive formula of the Chebyshev polynomials for every $m \geq 2$ (see $8.941, 2$ of \cite{GR})

\begin{equation*}
C_{m}^1(x) = 2xC_{m-1}^1(x) - C^1_{m-2}(x),
\end{equation*}
from which we have 
\begin{equation}\label{pm1}
\frac{1}{2} C_{m}^1(x) = xC_{m-1}^1(x) - \frac{1}{2} C^1_{m-2}(x).
\end{equation}
However, combining \eqref{cheby} with standard computations, we have that
\begin{equation}\label{pm1cos}
	\begin{aligned}
C_{m}^1(\cos t)  =\frac{\sin(mt)}{\sin t}\cos t + \cos(mt)=C_{m-1}^1(\cos t)\cos t+\cos(mt).
\end{aligned}
\end{equation}

Now, if we use \eqref{pm1} with $x=\cos t$, we obtain that 
\begin{equation}\label{pm2}
\frac{1}{2} C_{m}^1(\cos t) = C_{m-1}^{1}(\cos t) \cos t - \frac{1}{2} C_{m-2}^{1}(\cos t).
\end{equation}
Furthermore, combining \eqref{pm1cos} and \eqref{pm2}, we have that 
\begin{equation}\label{newformula}
C_{m}^1(\cos t) = 2\cos (mt) +C_{m-2}^1(\cos t).
\end{equation}
Now, we want to use \eqref{newformula} to evaluate $S_{k,m,4}$. Indeed, we have that
\begin{equation*}
S_{k,m, 4} = 2 \sum_{j=1}^{k-1} \cos \left( m\frac{2\pi j}{k} \right) + S_{k,m-2, 4}.
\end{equation*}
By using \eqref{sommatoria} we have that, if $m$ is a multiple of $k$
\begin{equation*}
\begin{aligned}
\alpha_{k,m, 4} & = m+1+S_{k,m, 4} = m-1+2k + S_{k,m-2, 4}  = 2k + \alpha_{k,m-2, 4}.
\end{aligned}
\end{equation*}
Afterwards, if $m$ is not a multiple of $k$, then by \eqref{sommatoria} we get 
\begin{equation*}
\alpha_{k,m, 4} = m+1+S_{k,m, 4}  = m -1 +S_{k,m-2} = \alpha_{k, m-2, 4}.
\end{equation*}
So, we have obtained that for any $m\geq 2$
\begin{equation}\label{alpha4}\alpha_{k,m, 4}:=\left\{\begin{aligned} &\alpha_{k,m-2, 4} +2k\quad &\mbox{if}\,\, m\,\, \mbox{is a multiple of}\,\, k\\
&\alpha_{k,m-2, 4}\quad &\mbox{if}\,\, m\,\, \mbox{is not a multiple of}\,\, k.\end{aligned}\right.\end{equation}
and hence \begin{equation}\label{alpha4uno}\alpha_{k,m, 4}\geq \alpha_{k,m-2, 4}\end{equation} for any $m\geq 2$.
In particular, by using \eqref{cheby} and \eqref{alpha4uno}
\begin{equation*}
\alpha_{k, 2m, 4} \geq \alpha_{k, 0, 4} = 1 + S_{k,0, 4} = 1 + \sum_{j=1}^{k-1} C_{0}^{1}\left( \cos \left( \frac{2\pi j}{k} \right) \right) = k;
\end{equation*}
instead,
\begin{equation*}
\alpha_{k,2m+1, 4} \geq \alpha_{k,1,4} = 2 + S_{k,1, 4} = 2+ 2 \sum_{j=1}^{k-1} \cos \left( \frac{2\pi j}{k} \right) =0.
\end{equation*}

Again, we minimize $r^2 Q_{m}(r)$, where now
\begin{equation*}
\begin{aligned}
r^2 Q_{m}(r) & = \frac{1}{2(m+1)(1-\rho^{2m+2})} \Bigg[ \frac{\rho^{2m+2}}{r^{2m+2}} -2\rho^{2m+2} + r^{2m+2} \Bigg].
\end{aligned}
\end{equation*}
As in the previous case $N=3$, we have that the minimum is achieved at $\sqrt\rho$ and
\begin{equation*}
\begin{aligned}
\rho Q_{m}(\sqrt{\rho}) & = \frac{1}{2(m+1)(1-\rho^{2m+2})} \left[2 \rho^{m+1} - 2\rho^{2m+2} \right]= \frac{\rho^{m+1}}{(m+1)(1+\rho^{m+1})}.
\end{aligned}
\end{equation*}
So
\begin{equation}\label{n4fin}
\begin{aligned}
\min_{r \in (\rho,1)} f(r)&=\frac{1}{2\omega_3}\left(2\sum_{m=0}^{+\infty} (m+1) \alpha_{k, m, 4}\rho Q_{m}(\sqrt{\rho}) -\mathfrak C_{k, 4}\right) \\
&\geq \frac{1}{2\omega_3}\left(2\sum_{m=0}^{+\infty} (2m+1) \alpha_{k, 2m, 4}\rho Q_{2m}(\sqrt{\rho})-\mathfrak C_{k, 4}\right)  \\
&\geq \frac{1}{2\omega_3}\left(2k\sum_{m=0}^{+\infty} \frac{\rho^{2m+1}}{1+\rho^{2m+1}}-\mathfrak C_{k, 4}\right)  \\& \geq \frac{1}{2\omega_3}\left(k \rho\sum_{m=0}^{+\infty} \rho^{2m}-\mathfrak C_{k, 4}\right) \\& =\frac{1}{2\omega_3}\left( k\frac{\rho}{1-\rho^2} -\mathfrak C_{k, 4}\right)>0 
\end{aligned}
\end{equation}
if $\rho$ is close to 1, 
since $$\frac{\rho}{1-\rho^2}  \xrightarrow{\rho \to 1^{-}} +\infty.$$
The thesis follows as before.
\subsubsection{The case $N\geq 5$}\label{N5subsec} In this case, in order to estimate $\alpha_{k, m, N}$, we have to use the following recursive formula for Gegenbauer polynomials (see for instance $22.7.23$ of \cite{AS})
\begin{equation}\label{pmlambdagen}
(m+\lambda) C_{m+1}^{\lambda-1}(x) = (\lambda-1)(C_{m+1}^{\lambda}(x) - C_{m-1}^{\lambda}(x)).
\end{equation}
Now, we recall that $\lambda=\frac{N-2}{2}$ (hence $\lambda-1=\frac{N-4}{2}$) and the definition of $S_{k, m, N}$ in \eqref{Skm} to rewrite \eqref{pmlambdagen} as 
\begin{equation}\label{recursiveN51}
S_{k,m+1, N} - S_{k, m-1, N}= \frac{2m+N-2}{N-4} S_{k, m+1, N-2}.
\end{equation}
Now, by using \eqref{alphakmN}, \eqref{ANm} and \eqref{recursiveN51} we get
\begin{equation*}
\begin{aligned}
\alpha_{k,m+1, N} - \alpha_{k,m-1, N} & = A_{N, m+1}-A_{N, m-1}+S_{k, m+1, N}-S_{k, m-1, N}
\\& = A_{N, m+1}-A_{N, m-1}+ \frac{2m+N-2}{N-4} S_{k,m+1, N-2}
\\& = A_{N, m+1}-A_{N, m-1} + \frac{2m+N-2}{N-4} \left( \alpha_{k,m+1, N-2} - A_{N-2, m+1}  \right) 
\\& = \underbrace{A_{N, m+1}-A_{N, m-1}- \frac{2m+N-2}{N-4} A_{N-2, m+1}  }_{(**)} + \frac{2m+N-2}{N-4} \alpha_{k,m+1, N-2}.
\end{aligned}
\end{equation*}
Now, with a straightforward computation, we can observe that $(**)=0$. 
Hence for any $m\in\mathbb N$
\begin{equation}\label{recursiveN5}
\alpha_{k,m+1, N} - \alpha_{k,m-1, N} = \frac{2m+N-2}{N-4} \alpha_{k,m+1, N-2}.
\end{equation}
By induction we can prove that for any $N\geq 3$ and any $m\geq 1$
\beq\label{sofferenza}\alpha_{k,m+1, N}-\alpha_{k, m-1, N}\geq 0.\eeq Indeed, for $N=5$ this follows directly by \eqref{recursiveN5}, since by \eqref{alpha3pari} and \eqref{alpha3dispari} we have $\alpha_{k,m+1,3}\geq 0$.\\ Now, we suppose that for any $m$ we have 
\begin{equation}\label{ipind}\alpha_{k,m+1, N-2} - \alpha_{k,m-1, N-2} \geq 0\end{equation} and we show that \eqref{sofferenza} holds. Indeed, by \eqref{recursiveN5} and by \eqref{ipind} we have 
$$\begin{aligned} \alpha_{k,m+1, N} - \alpha_{k,m-1, N}&\geq c_1\left\{\begin{aligned}&\alpha_{k, 1, N-2}\quad &\mbox{if}\,\, m\,\, \mbox{is even}\\
&\alpha_{k, 0, N-2}\quad &\mbox{if}\,\, m\,\, \mbox{is odd}.\end{aligned}\right.\end{aligned}$$
Now,  since $C_{1}^{\lambda}(x) = 2\lambda x$ (see $8.930, 2.10$ of \cite{GR}) , we can say that
\begin{equation*}
\begin{aligned}
\alpha_{k,1, N-2} &= A_{N-2, 1} + S_{k,1, N-2} = N-4 + \sum_{j=1}^{k-1}  C_{1}^{\frac{N-4}{2}} \left( \cos \left( \frac{2\pi j}{k} \right) \right)
\\& = (N-4) \left( 1+  \sum_{j=1}^{k-1} \cos \left( \frac{2\pi j}{k} \right) \right) = 0,
\end{aligned}
\end{equation*}
by using \eqref{sommatoria}. Similarly, we have
\begin{equation*}
\alpha_{k,0, N-2} = A_{N-2, 0}+ S_{k,0, N-2}  = 1 + \sum_{j=1}^{k-1}  C_{0}^{\frac{N-4}{2}} \left( \cos \left( \frac{2\pi j}{k} \right) \right)= k.
\end{equation*}
since $C_{0}^{\frac{N-4}{2}}(x) = 1$ (see $8.930, 1.10$ of \cite{GR}). Then, $\alpha_{k,m+1,N-2}\geq 0$, for any $m$ and \eqref{sofferenza} holds.\\ Again, we minimize $Q_{m}(r) r^{N-2}$ finding, as for the cases $N=3$ and $N=4$, a minimum point at $r=\sqrt{\rho}$. \ \\
In particular 
\begin{equation}\label{minQm}
(2m+N-2) \rho^{\frac{N-2}{2}} \cdot Q_{m}(\sqrt{\rho})   = \frac{2\rho^{\frac{2m+N-2}{2}}}{1+\rho^{\frac{2m+N-2}{2}}}.
\end{equation}
So
\begin{equation}\label{n5fin}
\begin{aligned}
\min_{r \in (\rho,1)}f(r)&=\frac{1}{\omega_{N-1}(N-2)}\left( \sum_{m=0}^{+\infty} (2m+N-2)\alpha_{k, m, N} Q_{m}(\sqrt{\rho}) \rho^{\frac{N-2}{2}} -\mathfrak C_{k, N}\right)\\
&= \frac{1}{\omega_{N-1}(N-2)}\left(\sum_{m=0}^{+\infty} \alpha_{k, m, N} \frac{2\rho^{\frac{2m+N-2}{2}}}{1+\rho^{\frac{2m+N-2}{2}}}-\mathfrak C_{k, N}\right)\\
&\geq \frac{1}{\omega_{N-1}(N-2)}\left(\sum_{m=0}^{+\infty} \alpha_{k, 2m, N} \frac{2\rho^{\frac{4m+N-2}{2}}}{1+\rho^{\frac{4m+N-2}{2}}}-\mathfrak C_{k, N}\right)\\
& \geq \frac{1}{\omega_{N-1}(N-2)}\left(k\sum_{m=0}^{+\infty} \rho^{\frac{4m+N-2}{2}}-\mathfrak C_{k, N}\right) \\
&= \frac{1}{\omega_{N-1}(N-2)}\left(k\rho^{\frac{N-2}{2}} \sum_{m=0}^{+\infty} \rho^{2m} -\mathfrak C_{k, N}\right)\\
&= \frac{1}{\omega_{N-1}(N-2)}\left(k\frac{\rho^{\frac{N-2}{2}}}{1-\rho^2} -\mathfrak C_{k, N}\right)>0
\end{aligned}
\end{equation}

if $\rho$ is close to $1$, since $$\frac{\rho^{\frac{N-2}{2}}}{1-\rho^2} \xrightarrow{\rho \to 1^{-}} +\infty.$$

Here we use also the fact that by using \eqref{recursiveN5} and \eqref{sofferenza} we get that

$$\alpha_{k, 2m, N}\geq \alpha_{k, 0, N}\geq k; \quad \alpha_{k, 2m+1, N}\geq \alpha_{k,1, N}\geq0.$$

\subsection{Negativity of $\Lambda_1(r)$}
Here we will show the following:\\\\
{\it For any integer $k$ there exists $\bar\rho_k\in (0, 1)$ such that for any $\rho\in (0,\bar\rho_k)$
it holds true that }
\begin{equation}\label{condmin1}
\min_{r\in (\rho, 1)} \Lambda_1(r)<0. \end{equation} From (iii) of \eqref{iii} this is equivalent to show that  $\displaystyle\min_{r\in(\rho, 1)}f(r)<0$.
\\In what follows we will use $f(r)$ in the form \eqref{fr}. 
Here we don't need to distinguish the different dimensions. \ \\
Before starting with the proof we introduce some preliminaries that regard the properties of the zonal harmonics.
\begin{proposition}(Proposition $5.27$, \cite{ABR})\label{axler}
Suppose $\xi, \eta$ such that $|\xi|=|\eta|=1$ and $m \geq 0$. Then
\begin{itemize}
\item[1.] $Z_{m}$ is real valued;
\item[2.] $Z_{m}(\xi, \eta)=Z_{m}(\eta, \xi)$;
\item[3.] $Z_{m}(\xi, T(\eta)) = Z_{m}(T^{-1}(\xi), \eta)$ for all $T \in O(n)$;
\item[4.] $Z_{m}(\eta, \eta) = d_{m}$;
\item[5.] $|Z_{m}(\eta, \xi)| \leq d_{m}$.
\end{itemize}
\end{proposition}
Here we have a technical estimate.
\begin{lemma}\label{rmk}
For all $N\geq 3$ and for all $m \geq 1$, we have that
\begin{equation*}
A_{N,m} \leq m^{N-3} \frac{(N-2)^{N-3}}{(N-3)!}.
\end{equation*}
\end{lemma}
\begin{proof}
By \eqref{ANm}
\begin{equation*}
\begin{aligned}
A_{N,m} & : = \binom{N+m-3}{N-3} = \frac{(N+m-3)!}{m!(N-3)!} = \frac{m! \displaystyle\prod_{l=1}^{N-3} (m+l)}{m! (N-3)!} \leq \frac{(m+N-3)^{N-3}}{(N-3)!}
\\& = \frac{m^{N-3}\left( 1+ \frac{N-3}{m} \right)^{N-3}}{(N-3)!} \leq m^{N-3} \frac{(N-2)^{N-3}}{(N-3)!}.
\end{aligned}
\end{equation*}
\end{proof}
By using $5$ of Proposition \ref{axler}, we have that
\begin{equation*}
\begin{aligned}
f(r) & = \frac{1}{\omega_{N-1}} \Bigg[ -\frac{\mathfrak C_{k, N}}{N-2}   + \sum_{m=0}^{+\infty} r^{N-2}Q_{m}(r) \left( d_m + \sum_{j=1}^{k-1} Z_{m} \left( \frac{\xi_1}{|\xi_1|}, \frac{\xi_{j+1}}{|\xi_{j+1}|} \right) \right) \Bigg]
\\& \leq \frac{1}{\omega_{N-1}} \Bigg[ - \frac{\mathfrak C_{k, N}}{N-2}  + k \sum_{m=0}^{+\infty} d_{m}r^{N-2} Q_{m}(r) \Bigg]\\
&=\frac{1}{\omega_{N-1}(N-2)}\Bigg[ -\mathfrak C_{k, N}  + k \sum_{m=0}^{+\infty} (N+2m-2) A_{N,m} r^{N-2} Q_{m}(r) \Bigg]
\end{aligned}
\end{equation*}
where the last equality follows from the definition of $d_m$ given in \eqref{binom}.\\
Now, if we are able to prove that, for $\rho \to 0$,
\begin{equation}\label{aim}
\min_{r \in [\rho,1]} \sum_{m=0}^{+\infty} (N+2m-2) A_{N,m} Q_{m}(r) r^{N-2} \to 0
\end{equation}
we get the claim. \ \\ To do so, we introduce the Eulerian numbers $A(n, m)$ which are the coefficients of the Eulerian polynomials. It is known that \beq\label{numeul}
\sum_{\ell=0}^{n-1}A(n, \ell)=n!\eeq Euler numbers appear in the generating function of sequences of $n$-th powers
\beq\label{succ}
\sum_{m=0}^{\infty}m^n x^m=\frac{\sum_{\ell=0}^{n-1}A(n, \ell)x^{\ell+1}}{(1-x)^{n+1}}, \quad\mbox{for}\,\, n\geq 0.\eeq

Now, using Lemma \ref{rmk}, \eqref{minQm}, \eqref{numeul}, \eqref{succ} (with $n=N-3$ and $x=\rho$) and letting $\gamma_{N} : = \frac{(N-2)^{N-3}}{(N-3)!}$, we can observe that 
\begin{equation}\label{negativo}
\begin{aligned}
&\min_{r \in [\rho,1]} \sum_{m=0}^{+\infty} (N+2m-2) A_{N,m} Q_{m}(r) r^{N-2} =\sum_{m=0}^{+\infty} (N+2m-2) A_{N,m} Q_{m}(\sqrt\rho) \rho^{\frac{N-2}{2}}\\
&\hskip+1.0cm = \sum_{m=0}^{+\infty} A_{N,m} \frac{2 \rho^{\frac{N+2m-2}{2}}}{1+\rho^{\frac{2m+N-2}{2}}} \leq 2 \rho^{\frac{N-2}{2}}\gamma_N\sum_{m=0}^{+\infty} m^{N-3} \rho^{m}\\
&\hskip+1.0cm =2 \rho^{\frac{N-2}{2}}\gamma_N\frac{\displaystyle\sum_{\ell=0}^{N-4}A(N-3, \ell)\rho^{\ell+1}}{(1-\rho)^{N-2}}<2 \rho^{\frac{N-2}{2}}\gamma_N\frac{\displaystyle\sum_{\ell=0}^{N-4}A(N-3, \ell)}{(1-\rho)^{N-2}}\\
&\hskip+1.0cm =2 \rho^{\frac{N-2}{2}}\gamma_N\frac{(N-3)!}{(1-\rho)^{N-2}}=2 (N-2)^{N-3}\frac{\rho^{\frac{N-2}{2}}}{(1-\rho)^{N-2}}\xrightarrow{\rho \to 0} 0.
\end{aligned}
\end{equation}


\section{Proof of Theorems \ref{main1}, \ref{main2} }\label{proof}
\begin{proposition}\label{main3}
For any $\rho\in (0, 1)$, there exists a unique critical point of $\Lambda_1(r)$ in $(\rho, 1)$, which is a global minimum. 
\end{proposition}
\begin{proof}
We recall that 
\begin{equation*}
\Lambda_1(r):=\frac{1}{\omega_{N-1}}\left[- \frac{1}{N-2} \frac{1}{r^{N-2}} \mathfrak C_{k,N}  + \sum_{m=0}^{+\infty} Q_{m}(r) c_{N,m} \alpha_{k, m, N}\right]
\end{equation*}
where $\alpha_{k, m, N}$ are defined in \eqref{alphakmN}. \\

We compute the first and second derivative of $\Lambda_1$ with respect to $r$, getting
\begin{equation*}
\Lambda'_{1}(r) : = \frac{1}{\omega_{N-1}} \Bigg[ \frac{\mathfrak C_{k,N}}{r^{N-1}} + \sum_{m=0}^{+\infty} c_{N,m}\alpha_{k, m, N} Q'_{m}(r) \Bigg]
\end{equation*}
and
\begin{equation}\label{lambda''}
\Lambda''_{1}(r) : = \frac{1}{\omega_{N-1}} \Bigg[ -\frac{(N-1)\mathfrak C_{k,N}}{r^{N}} + \sum_{m=0}^{+\infty} c_{N,m}\alpha_{k, m, N} Q''_{m}(r) \Bigg].
\end{equation}
Let $r_0$ be a critical point of $\Lambda_1(r)$ , then we have that 
\begin{equation}\label{primar0}
- \frac{\mathfrak C_{k,N}}{r^{N-1}_0} =  \sum_{m=0}^{+\infty} c_{N,m} \alpha_{k, m, N} Q'_{m}(r_0).
\end{equation} 
Now by using  \eqref{lambda''} and \eqref{primar0}
\begin{equation*}
\Lambda_1''(r_0)= \frac{1}{\omega_{N-1}} \sum_{m=0}^{+\infty} c_{N,m} \alpha_{k, m, N}  \underbrace{\Bigg[ (N-1) \frac{Q'_{m}(r_0)}{r_0} + Q''_{m}(r_0)}_{(I_1)} \Bigg].
\end{equation*}

\noindent Now, we observe that
\begin{equation*}
\begin{aligned}
Q'_{m}(r_0) & = \frac{1}{(2m+N-2)(1-\rho^{2m+N-2})} \Bigg[ -(2m+2N-4) \rho^{2m+N-2} r_0^{-2m-2N+3} 
\\& +2(N-2) \rho^{2m+N-2} r_0^{1-N} + (2m) r_0^{2m-1} \Bigg], 
\end{aligned}
\end{equation*}
while 
\begin{equation*}
\begin{aligned}
Q''_{m}(r_0) & = \frac{1}{(2m+N-2)(1-\rho^{2m+N-2})} \Bigg[ (2m+2N-4)(2m+2N-3)\rho^{2m+N-2} r_0^{-2m-2N+2} 
\\& -2(N-2)(N-1) \rho^{2m+N-2} r_0^{-N} + (2m)(2m-1) r_0^{2m-2} \Bigg].
\end{aligned}
\end{equation*}
Therefore
\begin{equation*}
\begin{aligned}
(I_1) & =  \frac{1}{(1-\rho^{2m+N-2})} \Bigg [(2N+2m-4) \rho^{2m+N-2} r_0^{-2N-2m+2} 
  + 2m r_0^{2m-2} \Bigg] 
\end{aligned}
\end{equation*}
and it is easy to see that $(I_1) >0$. \\
In sections \ref{N3subsec}, \ref{N4subsec} and \ref{N5subsec}, we proved that  $$\alpha_{k, m, N}\geq\left\{\begin{aligned} &0 \quad &\mbox{if}\,\, m\,\, \mbox{is odd}\\
&k \quad &\mbox{if}\,\, m\,\, \mbox{is even}.\end{aligned}\right.$$ Hence
$$\Lambda_1''(r_0)\geq \frac{1}{\omega_{N-1}}\sum_{m=0}^{+\infty} c_{N,2m} \alpha_{k, 2m, N}  \Bigg[ (N-1) \frac{Q'_{2m}(r_0)}{r_0} + Q''_{2m}(r_0) \Bigg]>0.$$ This means that $\Lambda_1(r)$ cannot have maximum points or inflection points. Thus  it has a unique critical point, which is a global minimum.

\end{proof}

\begin{proposition}\label{unicozero}
For any $k\geq 2$ there exists $\rho_k>0$ so that the function $$\mathfrak m: \rho\in (0, 1) \mapsto \min_{r\in (\rho, 1)}\Lambda_1(r)$$ has a unique zero in $\rho_k$.
\end{proposition}
\begin{proof}
By (ii) of \eqref{iii} it follows that $$\mathfrak m(\rho)=0\quad \iff\quad \min_{r\in(\rho, 1)} f(r)=0.$$ Moreover, in the previous section we have shown that there exists $\rho_k$ (near $1$) so that $\mathfrak m(\rho)>0$ when $\rho>\rho_k$ and there is also some $\bar \rho_k>0$ (near $0$) such that $\mathfrak m(\rho)<0$ when $\rho<\bar\rho_k$. By  \eqref{fr1}, we have that
\begin{equation}\label{h1}h(\rho)=\min_{r\in (\rho, 1)}f(r)=\frac{1}{\omega_{N-1}(N-2)}\left(\sum_{m=0}^{+\infty}\alpha_{k, m, N}\frac{2\rho^{\frac{2m+N-2}{2}}}{1+\rho^{\frac{2m+N-2}{2}}}-\mathfrak C_{k, N}\right).\end{equation}

\noindent It is simple to see (by using the non-negativity on $\alpha_{k, m, N}$ proved in the previous section) that $h(\rho)$ is strictly increasing. Hence $h(\rho)$ has a unique zero and so does $\mathfrak{m}$.
%
\end{proof}
\begin{remark}\label{soglia}
From Proposition \ref{unicozero} it follows that the threshold is uniquely determined, i.e. for any $k\geq 2$ there is a unique $\rho_k$ so that $$\min_{r\in(\rho, 1)}\Lambda_1(r)>0\quad\mbox{for}\,\, \rho>\rho_k\,\, \mbox{and}\,\, \min_{r\in(\rho, 1)}\Lambda_1(r)<0\quad\mbox{for}\,\, \rho<\rho_k.$$ Moreover $\rho_k\in(0, 1)$ is defined by the relation $ \mathfrak m(\rho_k)=0$, namely (recall \eqref{fr} and \eqref{minQm})
 \beq\label{sogliadef}
- \frac{\mathfrak C_{k,N} }{N-2}    + \sum_{m=0}^{+\infty}\frac{1}{2m+N-2}\frac{2\rho_k^{\frac{2m+N-2}{2}}}{1+\rho_k^{\frac{2m+N-2}{2}}} \left(d_m+\sum_{j=1}^{k-1}Z_{m} \left( \frac{\xi_1}{|\xi_1|}, \frac{\xi_{j+1}}{|\xi_{j+1}|} \right)\right)=0.
 \eeq

\noindent Although it is impossible to determine $\rho_k$ explicitly, we can estimate it from below. Indeed, arguing as in \eqref{negativo} we get
$$\begin{aligned}
f(\rho)&=- \frac{\mathfrak C_{k,N} }{N-2}    + \sum_{m=0}^{+\infty}\frac{1}{2m+N-2}\frac{2\rho^{\frac{2m+N-2}{2}}}{1+\rho^{\frac{2m+N-2}{2}}} \left(d_m+\sum_{j=1}^{k-1}Z_{m} \left( \frac{\xi_1}{|\xi_1|}, \frac{\xi_{j+1}}{|\xi_{j+1}|} \right)\right)\\
&\leq  \frac{1}{N-2}\underbrace{\left[2k(N-2)^{N-3}\frac{\rho^{\frac{N-2}{2}}}{(1-\rho)^{N-2}}-\mathfrak C_{k, N}\right]}_{:=\psi_k(\rho)},\end{aligned}$$ and $$\psi_k(\rho)=0\,\, \iff \,\, \frac{\rho}{(1-\rho)^2}=\left(\frac{\mathfrak C_{k, N}}{2k (N-2)^{N-3}}\right)^{\frac{2}{N-2}}:=\mathfrak a_{k, N}.$$ Hence it follows that \beq\label{contorhok}\rho_k >\frac{2\mathfrak a_{k, N}+1-\sqrt{4\mathfrak a_{k, N}+1}}{2\mathfrak a_{k, N}}.\eeq
The last quantity in the right hand side of \eqref{contorhok} goes to $1$ as $k\to+\infty$. Indeed,  we claim first that $\frac{\mathfrak C_{k, N}}{k}\to +\infty$ as $k\to+\infty$ since as $k\to+\infty$
 $$\frac{\mathfrak C_{k, N}}{k}>\frac{1}{k}\frac{1}{2^{N-2}}\sum_{j=1}^{k-1}\frac{1}{\sin\frac{\pi j}{k}}>\frac{1}{2^{N-2}\pi}\sum_{j=1}^{k-1}\frac 1 j \xrightarrow{k\to+\infty}+\infty.$$
Hence $\mathfrak a_{k, N}\to+\infty$ and $\rho_{k}\to 1$ as $k\to+\infty$.
\\ We remark that, when $N=4$, $\mathfrak C_{k, N}=\frac{k^2-1}{12}$ and hence $\frac{\mathfrak C_{k, N}}{k}$ is increasing in $k$.\\
In the general case, we are not able to prove that $\frac{\mathfrak C_{k, N}}{k}$ is increasing with $k$, but only that
$\mathfrak C_{k, N}$ which is increasing with $k$. Indeed, if $k$ is even
$$\begin{aligned}\mathfrak C_{k, N}&=\frac{1}{2^{N-2}}\sum_{j=1}^{\frac k 2-1}\frac{1}{\left(\sin\frac{\pi j}{k}\right)^{N-2}}+\frac{1}{2^{N-2}}+\frac{1}{2^{N-2}}\sum_{j=\frac k 2 +1}^{k-1}\frac{1}{\left(\sin\frac{\pi j}{k}\right)^{N-2}}\\
&=\frac{2}{2^{N-2}}\sum_{j=1}^{\frac k 2-1}\frac{1}{\left(\sin\frac{\pi j}{k}\right)^{N-2}}+\frac{1}{2^{N-2}}.
\end{aligned}$$  Instead, if $k$ is odd
$$\begin{aligned}\mathfrak C_{k, N}&=\frac{1}{2^{N-2}}\sum_{j=1}^{\frac{ k-1}{ 2}}\frac{1}{\left(\sin\frac{\pi j}{k}\right)^{N-2}}+\frac{1}{2^{N-2}}\sum_{j=\frac{ k-1}{ 2} +1}^{k-1}\frac{1}{\left(\sin\frac{\pi j}{k}\right)^{N-2}}=\frac{2}{2^{N-2}}\sum_{j=1}^{\frac{ k-1}{ 2}}\frac{1}{\left(\sin\frac{\pi j}{k}\right)^{N-2}}\end{aligned}$$ and the monotonicity easily follows.
\end{remark}

\begin{proof}[Proof of Theorems \ref{main1} and \ref{main2}]
Since for problem $(\mathcal{BN}^\e_+)$ in dimension $N=3$ and $N=4$ the proof is done respectively in \cite{MS} and \cite{PRV} we omit it, and we focus on the remaining cases. \\
Now, by Lemma \ref{lemma1} it follows that problems $(\mathcal P^\e_\pm)$ and $(\mathcal{BN}^\e_\pm)$ have a solution of the form \eqref{sol} if and only if the reduced functional $\mathfrak J_\e(d, r)$ has a critical point.\\\\
Let us consider problem $(\mathcal P^\e_-)$ with $N\geq 3$ and problem $(\mathcal{BN}^\e_+)$ with $N\geq 5$.\\ 
By Proposition \ref{exp}, this means to find a stable critical point of the function $\Psi$  (see \eqref{psi}) and this is easy to show provided $\Lambda_1(r)>0$.\\ By Proposition \ref{unicozero} and \eqref{condmin} it follows that for any $k\geq 2$ there exists $\rho_k>0$ such that  $\Lambda_1(r)>0$ in $(\rho, 1)$. \\\\ Moreover, by Proposition \ref{unicozero} it follows that for $\rho<\rho_k$, $\min_{r\in (\rho, 1)}\Lambda_1(r)<0$ and by Proposition \ref{main3} there is only one critical point. Hence there is no solution of the form \eqref{sol} of  $(\mathcal P^\e_-)$ with $N\geq 3$ and problem $(\mathcal{BN}^\e_+)$ with $N\geq 5$.\\\\
Let us consider problem $(\mathcal P^\e_+)$ with $N\geq 3$ and problem $(\mathcal{BN}^\e_-)$ with $N\geq 5$.  
Again we look for a stable critical point of the function $\Psi(d, r)$. In this case, by Proposition \ref{main3} and by \eqref{condmin1} there exists $\rho_k>0$ so that the unique minimum  $\min_{r\in(\rho, 1)}\Lambda_1(r)<0$ for $\rho<\rho_k$. Hence, let us choose $\rho <a_k<b_k<1$ so that $\Lambda_1(r)<0$ for any $r\in (a_k, b_k)$. Then, for $\mu>0$ fixed and sufficiently small, the following relations hold:
\begin{equation}\label{rel}
\frac{\partial}{\partial d} \Psi(\mu, r)>0\quad\mbox{and}\quad \frac{\partial}{\partial d} \Psi(\mu^{-1}, r)<0\quad \forall\,\, r\in [a_k, b_k]\end{equation} and 
\begin{equation}\label{rel1}
\frac{\partial}{\partial r} \Psi(d, a_k)<0\quad\mbox{and}\quad \frac{\partial}{\partial r} \Psi(d, b_k)>0\quad \forall\,\, d\in [\mu, \mu^{-1}].\end{equation}
Let us set $\mathcal R:= (\mu, \mu^{-1})\times (a_k, b_k)$ and let $(d_1, r_1)$ be the center point of $\mathcal R$.\\
Let us consider the homotopy $$H_t(d, r)=t\nabla \Psi(d, r)+(1-t)(-(d-d_1), r-r_1),\quad t\in [0, 1].$$ Then from \eqref{rel} and \eqref{rel1} we see that ${\rm deg}(H_t, \mathcal R, 0)$ is well defined and constant for $t\in [0, 1]$. It follows that ${\rm deg}(\nabla \Psi, \mathcal R, 0)=-1$. Since $\mathfrak J_\e$ is a small $C^1-$ uniform perturbation of $\Psi$, we conclude that ${\rm deg}(\nabla \mathfrak J_\e, \mathcal R, 0)=-1$ for all sufficiently small $\e$. The thesis easily follows.\\ Moreover for $\rho>\rho_k$, $\Lambda_1(r)>0$ and hence no solution of the form \eqref{sol} can be found.
\end{proof}

 \end{document}